\documentclass[letterpaper,11pt]{article}
\usepackage[margin=1in]{geometry}
\usepackage[T1]{fontenc}    
\usepackage{lmodern}
\usepackage{microtype}
\usepackage{multirow}
\usepackage{wrapfig}
\usepackage{booktabs}       
\usepackage{enumerate}
\usepackage{amsmath}
\usepackage{lipsum}
\usepackage{mathtools}
\usepackage{float}
\usepackage[dvipsnames,table,xcdraw]{xcolor}
\usepackage{bbm}
\usepackage{varioref}
\usepackage[hidelinks]{hyperref}
\usepackage{float}
\usepackage{placeins}
                                
\usepackage{amssymb} 
\usepackage{rotating}
\usepackage{graphicx}
\usepackage{algpseudocode}
\usepackage{mathrsfs}  
\usepackage{algcompatible}
\usepackage{amsmath}
\usepackage{cases}
\usepackage{graphicx}
\usepackage{multirow}
\usepackage{comment}
\usepackage{subcaption}
\usepackage[normalem]{ulem}

\usepackage{amsthm}
\DeclareUnicodeCharacter{00A0}{~}
\usepackage{algorithm}
\usepackage{algpseudocode}
\algnewcommand{\Initialize}[1]{%
	\State \textbf{Initialization:}
	\Statex {\raggedright #1}
}
\usepackage{xcolor} 

\usepackage{tcolorbox}

\usepackage{graphicx} 
\usepackage{amsmath}
\usepackage{amsthm}
\usepackage{amsfonts}
\usepackage{amssymb}
\usepackage{extpfeil}  
\usepackage{comment}
\usepackage{algorithm}
\usepackage{algpseudocodex}
\usepackage{enumerate}
\usepackage{enumitem}
\usepackage{booktabs, multirow} 
\usepackage{soul}
\usepackage{xcolor,colortbl} 
\usepackage{changepage,threeparttable} 
\usepackage[normalem]{ulem}
\usepackage{tcolorbox}
\usepackage{makecell}

\usepackage[textsize=footnotesize,textwidth=0.8in]{todonotes}

\definecolor{aqua}{rgb}{0.0, 0.5, 1.0}

\newcolumntype{L}[1]{>{\raggedright\arraybackslash}p{#1}}

\usepackage{xcolor}

\newcommand{\prox}{\mathrm{prox}}
\newcommand{\dom}{\mathrm{dom}}
\newcommand{\argmin}[1]{\operatorname*{argmin}_{#1}}

\def\R{\mathbb{R}}
\def\N{\mathbb{N}}

\def\E{\mathbb{E}}

\def\tf{\tilde{f}}
\def\tg{\tilde{g}}
\def\sF{\mathcal{F}}
\def\Real{\mathbb{R}}
\usepackage{bm}
\usepackage{tcolorbox}
\newcommand{\bxi}{\boldsymbol{\xi}}

\allowdisplaybreaks
\usepackage{tikz}
\usetikzlibrary{decorations.pathreplacing}

\NewDocumentEnvironment{alignb}{b}{%
  \begin{align*}
  \refstepcounter{equation} #1 \tag{\theequation}
  \end{align*}
}{\ignorespacesafterend}

\allowdisplaybreaks

\usepackage{pgfplots}
\pgfplotsset{compat=1.18}

\theoremstyle{plain}%

\newtheorem{theorem}{Theorem}

\newtheorem{lemma}{Lemma}
\newtheorem{remark}{Remark}

\newtheorem{assumption}{Assumption}

\usepackage[sort&compress]{natbib}
 \bibpunct[, ]{[}{]}{,}{n}{}{,}%

\title{A Stochastic LineseArch Method (SLAM) for
Minimizing Expectation Residuals: Weakening Boundedness of Gradient}
 
\author{Qi Wang\thanks{Department of Industrial and Operations Engineering, University of Michigan, United States. {\tt\small  qiwangqi@umich.edu}} \and Uday V. Shanbhag\thanks{Department of Industrial and Operations Engineering, University of Michigan, United States. {\tt\small  udaybag@umich.edu}
} \and Yue Xie\thanks{Musketeers Foundation Institute of Data Science and Department of Mathematics, University of Hong Kong, China. {{\tt\small yxie21@hku.hk}}
} 
}
\date{}

\begin{document}
\pagestyle{plain}
\maketitle
\begin{abstract}
We establish convergence guarantees for the Stochastic LineseArch Method
(\textbf{SLAM}) proposed and analyzed in \cite{Wang2025} for nonconvex composite
stochastic optimization, retaining the iteration and sample complexity orders of
the original analysis assuming neither the boundedness of  true gradients nor the Lipschitz continuity of the convex nonsmooth term $r$. \texttt{ChatGPT} was used to assist in
developing a key intermediate result. Under uniformly Lipschitz continuous and unbiased sampled
gradients 
with uniformly bounded variance with
mini-batch sizes proportional to the iteration budget, \textbf{SLAM}
achieves an iteration complexity $\mathcal{O}(\epsilon^{-2})$ and a
sample complexity $\mathcal{O}(\epsilon^{-4})$ for obtaining an iterate whose
expected proximal-gradient residual norm is at most $\epsilon$.
\end{abstract}


\section{Introduction}\label{sec:intro}
Consider minimizing the nonsmooth function $\phi$ over $\R^n$, as specified by 
\begin{align}
\min_{x \in \R^n} \ \phi(x) \, \triangleq \, f(x) + r(x), \quad \text{where } f(x) \, \triangleq \, \E \left[\, F(x, \bxi)\, \right], \label{eq:composite_opt}
\end{align}
where $\bxi: \Omega \to \Xi \subseteq  \mathbb{R}^d$ denotes a random
variable with realizations denoted by $\xi$,  $\Xi \,
\triangleq \, \left\{ \, \xi(\omega) \, \mid \, \omega \, \in
\, \Omega\,\right\}$,  $F: \Real^{n}  \times \Xi \to \Real$
is a real-valued function, $(\Omega,\mathcal{F},\mathbb{P})$ denotes the associated
probability space, and $r:\R^n \to \R\cup \left\{\, +\infty\, \right\}$ is a proper, 
closed, and convex function. In our recent work~\cite{Wang2025} we present a Stochastic LineseArch Method (\textbf{SLAM}), where we propose and analyze a proximal stochastic gradient method equipped with  a linesearch. The scheme is equipped with an  iteration complexity $\mathcal{O}(\epsilon^{-2})$ and sample complexity $\mathcal{O}(\epsilon^{-4})$ for obtaining an iterate whose expected norm of the proximal-gradient residual is at most $\epsilon$. However, our original analysis requires two additional assumptions on problem~\eqref{eq:composite_opt}: (i) \underline{\em Bounded gradients:} the true gradient $\nabla f$ is uniformly bounded on $\dom(r)$, and (ii) \underline{\em Lipschitz continuity of $r$:} the nonsmooth term $r$ is Lipschitz continuous on $\dom(r)$. Naturally, weakening both assumptions allows for contending with a far greater range of problems, motivating this note. 

The main purpose of this paper is to establish convergence guarantees for \textbf{SLAM}~\cite{Wang2025} without the two aforementioned assumptions, thereby extending its theoretical applicability to a broader class of nonconvex composite stochastic optimization problems while preserving the iteration and sample complexity guarantees  
presented in the original analysis. Moreover, the resulting complexity bounds hold for any prescribed sequence of positive integer stepsize-resetting periods, without requiring the period lengths to depend on the iteration budget.\\

\noindent \textbf{Use of AI.} We submitted our paper to ChatGPT (ChatGPT 6 Astra) to assist in developing the proof; it provided the main idea underlying Lemma~\ref{lem.gptproof}. Naturally the authors are responsible for verifying all mathematical arguments and for the final statements and proofs. The convergence guarantees and complexity bounds in Theorem~\ref{thm:convergence} follow directly from this lemma.\\

\textbf{Notation} We use $\R$ and $\N$ to denote the set of real
numbers and the set of nonnegative integers, respectively. We use $\dom(r)$
to denote the effective domain of function $r$, i.e., $\dom(r) \triangleq
\{x \in \R^n \mid r(x) < \infty\}$.  $\|\bullet\|$ denotes the Euclidean norm,
$\lceil \bullet \rceil$ denotes the ceiling function.

\section{Assumptions and Algorithm} \label{sec:2}
 
Our convergence analysis relies on Assumptions~\ref{ass:prox}, \ref{ass: F}, and \ref{ass: var}. Assumptions~\ref{ass: F} and \ref{ass: var} coincide with \cite[Assumptions~2 and~3]{Wang2025}. In contrast, \textbf{Assumption~\ref{ass:prox} requires neither Lipschitz continuity of $r$ nor boundedness of the true gradient $\nabla f$}, both of which are required in \cite[Assumption~1]{Wang2025}. The analysis therefore applies to a broader class of problems than that considered in \cite{Wang2025}.
\begin{assumption}\label{ass:prox} \em
The function $\phi$ in \eqref{eq:composite_opt} satisfies the following properties.
\begin{enumerate}
    \item[(i)] The function $r:\R^n \to (-\infty, \infty]$ is proper, closed, and convex.
    \item[(ii)]  The function $f$ is proper and closed, $\dom(f)$ is convex, $\dom(r) \, \subseteq \, \mathrm{int(\dom(}f))$, and $f$ is continuously differentiable on $\mathrm{int(\dom(}f))$.  
    \item[(iii)] The function $\phi$ is lower bounded by $\phi_{\inf}$ on $\dom(r)$, where $\phi_{\inf} \, \triangleq \, \inf \left\{ \, \phi(x) \, \mid \, x \in \dom(r)\, \right\}$.$\hfill \Box$
\end{enumerate}
\end{assumption}

\begin{assumption}\label{ass: F} \em
    For every $\xi \in \Xi$, $F(\bullet,\xi): \mathrm{int(\dom(}f)) \to \R$ is continuously differentiable and its gradient is Lipschitz continuous with constant $L$ on $\dom(r)$. $\hfill \Box$
\end{assumption}

In the next assumption, we provide a definition and assumption on the stochastic zeroth and first-order oracles.
\begin{assumption}[Stochastic zeroth and first-order oracles] \label{ass: var}
    For any given realization $\tilde{\xi} \in \Xi$ and $x \in \mbox{dom}(r)$, the stochastic oracles provide $F(x,\tilde{\xi})$ and $\nabla F(x,\tilde{\xi})$. In addition, the oracle also generates the map $F(\bullet,\tilde{\xi})$, i.e., the  oracle provides $F(u,\tilde{\xi})$ for any $u \, \in \, \mbox{dom}(r)$.
Moreover, there exists a constant $\sigma > 0$ such that, for every $x \in \dom(r)$, the following hold almost surely.
    \begin{align*}
        \E\left[\, F(x,\bxi)\mid x \, \right] \, = \, f(x), \,\, \E\left[\, \nabla F(x,\bxi)\mid x \,\right] \, = \, g(x), \,\, \text{and} \,\, \E\left[\, \| \nabla F(x,\bxi) - g(x) \|^2 \mid x\, \right] \, \le \, \sigma^2. 
    \end{align*} $\hfill \Box$
\end{assumption}

We  denote the (true) gradient map by $g$, i.e., $g(x) \triangleq \nabla f(x)$ for any $x \, \in \, \mbox{dom}(r)$.  Following~\cite{Wang2025}, at iteration $k=0,1,\ldots$, we use a batch of $N_k$ i.i.d. samples $\xi_{k,1},\ldots,\xi_{k,N_k}$, where $N_k$ is a positive integer and the batch is drawn independently of the history $\sF_k$ defined below. The sample-mean objective and gradient are
\begin{align}\label{sampledfg}
    \tf_k(\bullet) \, \triangleq \, \frac{1}{N_k} \sum_{i = 1}^{N_k} F(\bullet,\xi_{k,i})\quad \text{ and } \quad \tg_k(\bullet) \, \triangleq \, \frac{1}{N_k} \sum_{i = 1}^{N_k} \nabla F(\bullet,\xi_{k,i}),
\end{align}
respectively. Correspondingly, $\tilde{\phi}_k$ is defined as $\tilde{\phi}_k(\bullet) \triangleq \tf_k(\bullet) + r(\bullet)$.
At the current iterate $x_k\in\dom(r)$, Assumption~\ref{ass: var} implies that the following hold almost surely:
    \begin{align*}
        \E\left[\, \tf_k(x_k)\, \big|\, x_k\, \right] \, & = \, f(x_k),\  \E\left[\tg_k(x_k)\, \big|\, x_k\right] = g(x_k), \text{ and}\;  \\
        \E\left[\, \|\tg_k(x_k) - g(x_k)\|^2\, \big|\, x_k\, \right] \, & \le \, \frac{\sigma^2}{N_k}. \,\, 
    \end{align*}
We employ a proximal gradient method to handle the nonsmooth term $r$ in \eqref{eq:composite_opt}. For $t > 0$, the proximal operator of $tr$ is defined as
\begin{align*}
    \prox_{tr}(x) \triangleq  \argmin{y \in \R^n} \left\{ r(y) + \frac{1}{2t} \| y - x \|^2 \right\} = \argmin{y \in \R^n} \left\{ tr(y) + \frac{1}{2} \| y - x \|^2 \right\}.
\end{align*}
We reproduce \textbf{SLAM}~\cite[Algorithm~1]{Wang2025} in Algorithm~\ref{alg: sls pk prox} for completeness. At iteration $k$, \textbf{SLAM} employs a backtracking procedure to determine the stepsize $t_k$, recomputing the trial point $x_k(t_k)$, defined as  \begin{align} \label{df:x_k}
x_k(t_k) \, \triangleq\, \prox_{t_k r}(x_k - t_k \tg_k(x_k)),
\end{align}
until the following condition holds:
\begin{align}
    \tilde{\phi}_k(x_{k}(t_k) ) - \tilde{\phi}_k (x_k) \le - \frac{\alpha}{t_k} \|x_k - x_k(t_k)\|^2. \label{eq.prox.armijo}
\end{align}
As in \cite{Wang2025}, for a point $x \in \dom(r)$ and a stepsize $t > 0$, we define the stationarity residual as
\begin{align} \label{eq.G.prox}
    G_{t}(x) \, \triangleq \, \frac{1}{t}\left(\, x - \prox_{tr}(x - t g(x))\, \right)
\end{align}
while its stochastic counterpart, based on using a sample mean $\tilde{g}_k(x)$ constructed using a batch of size $N_k$ drawn at iteration $k$, as
\begin{align*}
    \tilde G^k_{t}(x) \, \triangleq \, \frac{1}{t}\left(\, x - \prox_{tr}(x - t \tg_k(x))\, \right).
\end{align*}
With the above notation, the stochastic Armijo condition \eqref{eq.prox.armijo} at step $k$ is equivalent to 
\begin{align} \label{eq.prox.armijo.G}
    \tilde \phi_k(x_{k}(t_k)) - \tilde \phi_k(x_k) \, \le \, -\alpha t_k \left\|\, \tilde G^k_{t_k}(x_k)\, \right\|^2.
\end{align}
We will derive convergence guarantees and rates for the stationarity metric  $\|G_s(x_{R_K})\|$, where $R_K$ denotes an index randomly selected via a uniform distribution from
$\left\{\, 0,1,\cdots,K-1\, \right\}$, while $s$ denotes the user-defined initial steplength. Finally, we define the history $\sF_k$  generated by the Algorithm~\ref{alg: sls pk prox} at iteration $k$ as 
\begin{align*}
    \sF_0 \, \triangleq \, \sigma(x_0)\, \text{ and }\,\, \sF_k \, \triangleq \, \sigma\left( \, x_0,\, (\xi_{0,j})_{j=1}^{N_0}, (\xi_{1,j})_{j=1}^{N_1}, \cdots, (\xi_{k-1,j})_{j=1}^{N_{k-1}} \, \right), \quad k \ge 1.
\end{align*}

\begin{tcolorbox}
\begin{algorithm}[H]\caption{Stochastic LineseArch  Method ({\bf SLAM}) \cite[Algorithm~1]{Wang2025}}\label{alg: sls pk prox}
    \begin{algorithmic}[1]
        \Require $x_0 \in \dom(r),~\alpha \in (0,1),~\beta \in (0,1),~\{p_j\}$ where $p_j \in \N$, $s \in {(0, \infty)},~\{N_k\}$ where $N_k \in \N$.
        \For{ $k\in\{0, 1, \ldots, K-1\}$}
        \State Generate samples $\{\xi_{k,j}\}_{j=1}^{N_k}$ and define $ \tf_k(\bullet)$ and $\tg_k(\bullet)$ as in \eqref{sampledfg};
        \If{$k$ initiates a new period in $\{p_j\}$} 
        \State Set $t_{k}^{\mathrm{init}} \gets s$;
        \Else
        \State Set $t_{k}^{\mathrm{init}} \gets t_{k-1}$;  \label{line.init_old}
        \EndIf 
        \State Set $t_k \gets t_{k}^{\mathrm{init}}$;
        \While{ {\eqref{eq.prox.armijo} is not satisfied}} \label{line.backtrack}
        \State Set $t_k \gets \beta t_k$; \label{line.prox backtrack}
        \EndWhile
        \State Update $x_{k+1} \gets x_k(t_k)$; 
        \EndFor
    \end{algorithmic}
\end{algorithm}
\end{tcolorbox}
\begin{remark}[Motivation of designing the period sequence $\{p_j\}$]
    To reduce the cost of objective evaluations during backtracking at Line~\ref{line.backtrack} of Algorithm~\ref{alg: sls pk prox}, \textbf{SLAM} reuses $t_{k-1}$ as the initial stepsize at iteration $k$, as in Line~\ref{line.init_old}. However, repeatedly reusing the accepted stepsize may prevent the method from trying larger steps. A stepsize reduced at an earlier iteration may be unnecessarily conservative at later iterates, where a larger stepsize may satisfy the sampled Armijo condition and allow faster progress. The prescribed period sequence $\{p_j\}$ allows \textbf{SLAM} to restart the linesearch from $s$ at the beginning of each period.
\end{remark}
\begin{remark}[Input parameters of Algorithm~\ref{alg: sls pk prox}]
    The requirements on $x_0$, $\alpha$, $\beta$, $s$, and $\{N_k\}$ in Algorithm~\ref{alg: sls pk prox} are the same as in \cite{Wang2025}. In particular, the complexity result uses a fixed batch size proportional to the iteration budget $K$. However, \textbf{our analysis allows any prescribed sequence of positive integer stepsize-resetting period lengths $\{p_j\}$}. In contrast, the iteration and sample complexity bounds in \cite{Wang2025} depend on the period lengths, whose choices must be tied to the iteration budget to obtain the stated complexity orders.
\end{remark}
\section{Convergence analysis} \label{sec:3}

Lemma~\ref{lm.prox:t_lb} and parts~(i) and~(ii) of Lemma~\ref{lm:bounds_G} were proven in \cite{Wang2025} without using boundedness of the true gradient or Lipschitz continuity of $r$. We restate these results for completeness.
\begin{lemma}[Bounds on $t_k$]\label{lm.prox:t_lb}\em
    Consider the sequences generated by Algorithm~\ref{alg: sls pk prox} and suppose that Assumptions~\ref{ass:prox}, \ref{ass: F}, and \ref{ass: var} hold. Then, for any $k \in \N$, 
    \begin{align}
        \min\left\{\, s,\frac{2(1-\alpha)\beta}{L}\, \right\} \, \triangleq \,  c \, \le \, t_k \, \le \, s. \label{eq:prox.t_bound}
    \end{align}
\end{lemma}
\begin{proof}
    See proof of  \cite[Lemma~3]{Wang2025}.
\end{proof}

\begin{lemma}\label{lm:bounds_G} \em Consider Algorithm~\ref{alg: sls pk prox}. Suppose that Assumption~\ref{ass:prox} holds and $k \in \N$ and $x \in \dom(r)$.    
\begin{enumerate}[label=(\roman*)]
    \item \label{lem3.i} For any $\eta_1 \ge \eta_2 > 0$, we have that 
    \begin{align} \label{rand-resid1}
        \left\|\, \tilde{G}^k_{\eta_2}(x)\, \right\| \, \ge \, \left\|\, \tilde{G}^k_{\eta_1}(x)\, \right\|
    \end{align}
    and 
    \begin{align}\label{rand-resid2}
        \eta_2\left\|\, \tilde{G}^k_{\eta_2}(x)\, \right\| \, \le \, \eta_1\left\|\, \tilde{G}^k_{\eta_1}(x)\, \right\|.
    \end{align}
    The analogs of \eqref{rand-resid1} and \eqref{rand-resid2} also hold for the true residuals $G_{\eta_1}$ and $G_{\eta_2}$, respectively.
    \item \label{lem3.ii} For any $t > 0$,
    \begin{align*}
        - 2\|\tilde{G}^k_t(x) \|^2 & \le -\| \, G_t(x) \, \|^2 + 2 \| g(x) -  \tilde{g}_k(x)\|^2  .  \end{align*}
\end{enumerate}
\end{lemma}
\begin{proof}
    See proof of \cite[Lemma~5]{Wang2025}
\end{proof}

The following lemma provides an alternative to \cite[Proposition~1-2]{Wang2025} without the bounded true-gradient assumption or the assumption of Lipschitz continuity of the nonsmooth term. 
\begin{lemma}\label{lem.gptproof}
Consider the sequences generated by Algorithm~\ref{alg: sls pk prox} and suppose that Assumptions~\ref{ass:prox}, \ref{ass: F}, and \ref{ass: var} hold. Define $\Gamma$ as
    \begin{align*}
        \Gamma \, \triangleq \, \alpha c + \sqrt{1 + \lfloor \log_{\beta}(c/s)\rfloor} s + Ls^2 + \frac{(1 + \lfloor \log_{\beta}(c/s)\rfloor) s^2}{\alpha c}.
    \end{align*}
    Then for any $k \in \N$, the following holds.
\begin{align*}
        \E\left[\, \phi(x_{k+1}) - \phi(x_k) \, \mid \, \mathcal{F}_k \, \right] &\le - \frac{\alpha  c \left\|\, G_s(x_k)\, \right\|^2 }{4}  +  \frac{\Gamma \sigma^2}{N_k}.
\end{align*}
\end{lemma}
\begin{proof}
 From the sampled linesearch condition \eqref{eq.prox.armijo.G}, one has
 \begin{align*}
        & \phi(x_{k+1}) - \phi(x_k) \\
         \, &\le \, -\alpha t_k \left\|\, \tilde G^k_{t_k}(x_k)\, \right\|^2 + \left( f(x_{k+1}) - \tf_k(x_{k+1})\right) - \left( f(x_k) - \tf_k(x_k) \right) \\
        &\le -\alpha c \left\|\, \tilde G^k_{s}(x_k)\, \right\|^2 + \underbrace{\left( f(x_{k+1}) - \tf_k(x_{k+1})\right)}_{\, \triangleq \, h_k(x_{k+1})\,} - \underbrace{\left( f(x_k) - \tf_k(x_k) \right)}_{\, \triangleq \, h_k(x_{k})\,} \\
        &\le - \frac{\alpha c}{2} \left\|\, G_s(x_k)\, \right\|^2 + \alpha c \|g(x_k) - \tg_k(x_k)\|^2 + h_k(x_{k+1}) - h_k (x_k)
 \end{align*}
The second inequality uses $c\le t_k\le s$ from Lemma~\ref{lm.prox:t_lb} and the the monotonicity of the stationarity residual claimed in Lemma~\ref{lm:bounds_G}(i). The third inequality uses Lemma~\ref{lm:bounds_G}(ii) and the definition $h_k(z)\triangleq f(z)-\tf_k(z)$. Taking conditional expectation given $\mathcal{F}_k$, using Assumption~\ref{ass: var} and the $\mathcal{F}_k$-measurability of $x_k$, yields
  \begin{align}
    \E\left[\, \phi(x_{k+1}) - \phi(x_k) \, \mid \, \mathcal{F}_k \, \right] \, \le \, - \frac{\alpha c}{2} \left\|\, G_s(x_k)\, \right\|^2  + \frac{\sigma^2 \alpha c}{N_k} + \E\left[ \,  h_k(x_{k+1}) - h_k (x_k) \, \mid \, \mathcal{F}_k\, \right]. \label{eq.phi decrease}
 \end{align}
We now bound the last term in \eqref{eq.phi decrease}. For each fixed $t>0$, define the $\mathcal{F}_k$-measurable auxiliary point
 \begin{align*}
    y_k(t) \triangleq  \prox_{tr}(x_k - t g(x_k)).
 \end{align*}
Thus, $y_k(t)$ is the proximal update using the true gradient at $x_k$, given by $g(x_k)$. We decompose the error as
\begin{align}\label{ineq:h_k}
    \notag \E\left[ \,  h_k(x_{k+1}) - h_k (x_k) \, \mid \, \mathcal{F}_k\, \right] & = \E\left[ \, h_k(x_{k+1}) - h_k(y_k(t_k))\, \mid \, \mathcal{F}_k\, \right]  \\
        & ~~~+  \E\left[\,   h_k(y_k(t_k)) - h_k (x_k) \, \mid \, \mathcal{F}_k \,\right].
\end{align}
By Assumptions~\ref{ass: F} and~\ref{ass: var}, both $f$ and $\tf_k$ have $L$-Lipschitz continuous gradients on $\dom(r)$, so $h_k$ is $2L$-smooth there. Hence,
 \begin{alignb}
    h_k(x_{k+1}) - h_k(y_k(t_k)) &\le \nabla h_k(y_k(t_k))^\top (x_{k+1} - y_k(t_k)) + L \|x_{k+1} - y_k(t_k)\|^2 \\
&\le \|\nabla h_k(y_k(t_k))\| \|x_{k+1} - y_k(t_k)\| + L \|x_{k+1} - y_k(t_k)\|^2.  \label{eq.hk lipschitz}
 \end{alignb}
Nonexpansiveness of the proximal operator and $t_k\le s$ give
 \begin{align}
    \|x_{k+1} - y_k(t_k)\| \le t_k \|g(x_k) - \tg_k(x_k)\| \le s\|g(x_k) - \tg_k(x_k)\|. \label{eq.yktk diff bound}
 \end{align}
To bound the term $\|\nabla h_k(y_k(t_k))\|$, note that Lemma~\ref{lm.prox:t_lb} and the backtracking rule imply that $t_k$ belongs to the fixed finite set, given by
 \begin{align}\label{df:T}
    \mathcal{T} \, \triangleq\,  \left\{\, s\beta^j \, \mid \, j \in \N \text{ and } s\beta^j \ge c\, \right\},
 \end{align} 
 and the cardinality of $ \mathcal{T}$ is given by $|\mathcal{T}| = 1 + \lfloor \log_{\beta}(c/s)\rfloor$.
 Therefore, 
 \begin{align}
    \left\|\, \nabla h_k(y_k(t_k))\, \right\|^2 \, \le \, \max_{t \in \mathcal{T}} \left\|\, \nabla h_k(y_k(t))\, \right\|^2 \, \le \, \sum_{t \in \mathcal{T}}  \left\|\, \nabla h_k(y_k(t))\, \right\|^2. \label{eq.htk gradient bound}
 \end{align}
Taking conditional expectation given $\mathcal{F}_k$ in \eqref{eq.hk lipschitz}, applying the Cauchy--Schwarz inequality, and using \eqref{eq.yktk diff bound}, \eqref{eq.htk gradient bound}, and $\nabla h_k=g-\tg_k$, we obtain
 \begin{alignb}
    & \E\left[ h_k(x_{k+1}) - h_k(y_k(t_k))| \mathcal{F}_k\right] \\
    &\le  \E\left[  \|\nabla h_k(y_k(t_k))\| \|x_{k+1} - y_k(t_k)\| | \mathcal{F}_k\right] + \E\left[ L \|x_{k+1} - y_k(t_k)\|^2 | \mathcal{F}_k\right] \\
    & \le \sqrt{ \E\left[  \|\nabla h_k(y_k(t_k))\|^2 | \mathcal{F}_k\right] \E\left[  \|x_{k+1} - y_k(t_k)\|^2 | \mathcal{F}_k\right] } + L \E\left[ \|x_{k+1} - y_k(t_k)\|^2 | \mathcal{F}_k\right] \\
    & \le \sqrt{ \E\left[  \sum_{t \in \mathcal{T}}  \|\nabla h_k(y_k(t))\|^2 | \mathcal{F}_k\right] \E\left[  \|x_{k+1} - y_k(t_k)\|^2 | \mathcal{F}_k\right] } + L \E\left[ \|x_{k+1} - y_k(t_k)\|^2 | \mathcal{F}_k\right] \\
    & \le \sqrt{ \left(  \sum_{t \in \mathcal{T}}  \E[\|\nabla h_k(y_k(t))\|^2 | \mathcal{F}_k]\right) s^2\E\left[  \|g(x_k) - \tg_k(x_k)\|^2 | \mathcal{F}_k\right] } + Ls^2 \E\left[ \|g(x_k) - \tg_k(x_k)\|^2| \mathcal{F}_k\right] \\
    & \le \sqrt{ \left(  \sum_{t \in \mathcal{T}}  \underbrace{\E[\|g(y_k(t)) - \tg_k(y_k(t))\|^2 | \mathcal{F}_k]}_{\, \le \, \sigma^2/N_k}\right) s^2\underbrace{\E\left[  \|g(x_k) - \tg_k(x_k)\|^2 | \mathcal{F}_k\right]}_{\, \le \, \sigma^2/N_k} } + Ls^2 \E\left[ \|g(x_k) - \tg_k(x_k)\|^2| \mathcal{F}_k\right] \\
    & \le \frac{\sqrt{ |\mathcal{T}|} s\sigma^2}{N_k}  + \frac{Ls^2\sigma^2}{N_k}. \label{eq.hkyktk xkp1 bound}
 \end{alignb}
 We now bound the second term in \eqref{ineq:h_k}, given by $\E\left[ h_k(y_k(t_k)) - h_k(x_k) \mid \mathcal{F}_k\right]$. For each fixed $t \in \mathcal{T}$, the points $x_k$ and $y_k(t)$ are $\mathcal{F}_k$-measurable and belong to $\dom(r)$. Hence, $(1-u)x_k+u y_k(t)$ is $\mathcal{F}_k$-measurable for each fixed $u \in [0,1]$ and belongs to $\dom(r)$ by convexity of dom$(r)$. Since the current batch is independent of $\mathcal{F}_k$ and $\nabla h_k=g-\tg_k$, Assumption~\ref{ass: var} gives
 \begin{align*}
     \E\left[\|\nabla h_k((1-u)x_k+u y_k(t))\|^2 \mid \mathcal{F}_k\right]
     \le \frac{\sigma^2}{N_k}
 \end{align*}
for any fixed $t \in \mathcal{T}$ and $u \in [0,1]$. Let $z_k(u; y_k(t), x_k) \triangleq h_k((1-u)x_k+u y_k(t))$, and then by the chain rule, $z_k'(u; y_k(t), x_k) = \nabla h_k((1-u)x_k+u y_k(t)) ^\top (y_k(t) - x_k)$. By the fundamental theorem of calculus, one has
 \begin{align*}
     h_k(y_k(t))-h_k(x_k) & = z_k(1; y_k(t), x_k) - z_k(0; y_k(t), x_k) \\ 
     & = \int_0^1 z_k'(u; y_k(t), x_k)\, du\\
     & = \int_0^1 \nabla h_k((1-u)x_k+u y_k(t))^\top (y_k(t)-x_k)\,du.
 \end{align*}
 Applying Jensen's inequality on the integral and then the Cauchy-Schwarz inequality to the inner product in the integrand  yields
 \begin{align*}
     |h_k(y_k(t))-h_k(x_k)|^2
     &\le \int_0^1
         |\nabla h_k((1-u)x_k+u y_k(t))^\top (y_k(t)-x_k)|^2\,du \\
     &\le \|y_k(t)-x_k\|^2
         \int_0^1 \|\nabla h_k((1-u)x_k+u y_k(t))\|^2\,du.
 \end{align*}
Taking conditional expectation and using $y_k(t)-x_k=-tG_t(x_k)$ together with Lemma~\ref{lm:bounds_G}(i), we obtain
 \begin{alignb}
     &\E[|h_k(y_k(t))-h_k(x_k)|^2 \mid \mathcal{F}_k] \\
     & \le \E\left[\,\|y_k(t)-x_k\|^2
         \int_0^1 \|\nabla h_k((1-u)x_k+u y_k(t))\|^2\,du \, \mid \, \mathcal{F}_k \,\right]\\
    &  \, = \, \|y_k(t)-x_k\|^2\E\left[\,
         \int_0^1 \|\nabla h_k((1-u)x_k+u y_k(t))\|^2\,du \, \mid \, \mathcal{F}_k \,\right]\\
     &\overset{\rm Fubini}{\, = \,} \left\|\, y_k(t)-x_k\, \right\|^2 \int_0^1
         \E\left[\, \left\|\, \nabla h_k((1-u)x_k+u y_k(t))\, \right\|^2\, \mid \, \mathcal{F}_k\, \right]\,du \\
     &=  t^2 \left\|\, G_t(x_k)\, \right\|^2  \int_0^1
         \E\left[\, \left\|\, \nabla h_k((1-u)x_k+u y_k(t))\, \right\|^2\, \mid \, \mathcal{F}_k\, \right]\,du  \\
     &\le \frac{\sigma^2 s^2}{N_k}\left\|\, G_s(x_k)\, \right\|^2.
     \label{eq.hk fixed step second moment}
 \end{alignb}
 Now, using $t_k \in \mathcal{T}$, Jensen's inequality, and \eqref{eq.hk fixed step second moment}, one has
 \begin{alignb}
     \E\left[\, h_k(y_k(t_k))-h_k(x_k) \mid \mathcal{F}_k\, \right] 
     &\le \E\left[\, \max_{t \in \mathcal{T}} \left|\, h_k(y_k(t))-h_k(x_k)\, \right| \mid \mathcal{F}_k\, \right] \\
     &\le \left(\E\left[\max_{t \in \mathcal{T}} |h_k(y_k(t))-h_k(x_k)|^2
         \mid \mathcal{F}_k\right]\right)^{1/2} \\
     &\le \left(\E\left[\sum_{t \in \mathcal{T}} |h_k(y_k(t))-h_k(x_k)|^2
         \mid \mathcal{F}_k\right]\right)^{1/2} \\
         &\overset{\eqref{eq.hk fixed step second moment}}{\le} s\|G_s(x_k)\|\sqrt{\frac{|\mathcal{T}|\sigma^2}{N_k}} \\
     &\le \frac{\alpha c}{4}\|G_s(x_k)\|^2
         + \frac{|\mathcal{T}|\sigma^2 s^2}{\alpha c N_k},
     \label{eq.hkyktk bound}
 \end{alignb}
 where in the last inequality we use Young's inequality ,given by $ab \le \delta a^2 + b^2/(4\delta)$ for any $a,b \in \R$ and $\delta > 0$. 
 Now, combining \eqref{eq.phi decrease}, \eqref{eq.hkyktk xkp1 bound}, and \eqref{eq.hkyktk bound}, we obtain 
\begin{align*}
        \E\left[\, \phi(x_{k+1}) - \phi(x_k) \, \mid \, \mathcal{F}_k\, \right ] &\le - \frac{\alpha c \left\|\, G_s(x_k)\, \right\|^2 }{2}  + \frac{\sigma^2 \alpha  c}{N_k} + \frac{\sqrt{ |\mathcal{T}|} s\sigma^2}{N_k}  + \frac{Ls^2\sigma^2}{N_k} + \frac{\alpha c \|G_{s}(x_k)\|^2}{4} + \frac{| \mathcal{T}| \sigma^2 s^2}{\alpha c N_k} \\
        & = - \frac{\alpha  c \left\|\, G_s(x_k)\, \right\|^2 }{4} + \left(\alpha c + \sqrt{ |\mathcal{T}|} s + Ls^2 + \frac{| \mathcal{T}| s^2}{\alpha c}  \right) \frac{\sigma^2}{N_k}.
\end{align*}
The result follows by substituting $|\mathcal{T}| = 1 + \lfloor \log_{\beta}(c/s)\rfloor$ and using $\Gamma$ defined in the lemma.
\end{proof}

We observe that \texttt{ChatGPT} suggested two proof ideas that we had not previously considered: (i) the use of a deterministic finite set containing all possible values of the random stepsize $t_k$ to bound sampling errors over that set; and (ii) exploiting the $2L$-Lipschitz continuity of the gradient of $f-\tf_k$. The above proof represents a version of the proof that has been modified to provide clarity. The following theorem follows directly from Lemma~\ref{lem.gptproof}, providing convergence guarantees and complexity bounds comparable to those in \cite[Theorem~1]{Wang2025} but does so under weaker assumptions; specifically, we require neither the boundedness of true gradient map nor the Lipschitz continuity of $r$.
\begin{theorem}\label{thm:convergence}
    Consider the sequences generated by Algorithm~\ref{alg: sls pk prox} and suppose that Assumptions~\ref{ass:prox}, \ref{ass: F}, and \ref{ass: var} hold. Let $\Gamma$ be defined as in Lemma~\ref{lem.gptproof}.
    For any $K \in \N\setminus\{0\}$, choose $N_k = \lceil c_b K \rceil$, where $c_b > 0$. If $R_K$ is uniformly distributed on $\{0,1,\ldots,K-1\}$ and independent of the generated iterates, then
    \begin{align*}
    \E\left[\, \left\|\, G_s(x_{R_K})\, \right\|^2 \, \right]
    &\le \frac{4}{\alpha cK}\left(\phi(x_0) - \phi_{\inf} + \frac{\Gamma\sigma^2}{c_b}\right).
\end{align*}
Moreover, for any real number $\epsilon > 0$, we define the iteration complexity of computing an $\epsilon$-solution as the smallest number of iterations $K(\epsilon)$ that guarantees
\begin{align}\label{eq.convergence measure}
    \E\left[ \, \left\| \,  G_{s}(x_{R_{K(\epsilon)}})  \, \right\|\, \right] \, \le \, \epsilon,
\end{align}
while the sample complexity, denoted by  $S(\epsilon) \, \triangleq \, {\displaystyle \sum_{k=0}^{K(\epsilon)-1}} N_k$, represents the number of samples used to achieve the stated accuracy guarantee~\eqref{eq.convergence measure}. Then
\begin{align*}
K(\epsilon) \, &\le \, \frac{4\left(\phi(x_0) - \phi_{\inf} + \Gamma \sigma^2/c_b\right)}{\alpha c \epsilon^2} + 1
= \mathcal{O}(\epsilon^{-2}), \text{ and }\\
S(\epsilon) \, &\le \, c_b\left(\frac{4\left(\phi(x_0) - \phi_{\inf} + \Gamma \sigma^2/c_b\right)}{\alpha c \epsilon^2} + 1\right)^2  + \frac{4\left(\phi(x_0) - \phi_{\inf} + \Gamma \sigma^2/c_b\right)}{\alpha c \epsilon^2} + 1
= \mathcal{O}(\epsilon^{-4}).
\end{align*}
\end{theorem}
\begin{proof}
The first result follows by taking total expectation in Lemma~\ref{lem.gptproof}, rearranging, summing over $k=0,\ldots,K-1$, and using
\begin{align*}
\sum_{k=0}^{K-1} \frac{1}{N_k} \le \sum_{k=0}^{K-1} \frac{1}{c_b K} = \frac{1}{c_b}.
\end{align*}
To derive the iteration complexity, we invoke Jensen's inequality, whereby
\begin{align*}
    \E\left[\left\|G_s(x_{R_K})\right\|\right]
    &\le \sqrt{\E\left[\left\|G_s(x_{R_K})\right\|^2\right]} \le \sqrt{\frac{4\left(\phi(x_0) - \phi_{\inf} + \Gamma \sigma^2/c_b\right)}{\alpha c K}}.
\end{align*}
Thus, for any positive integer $K$, satisfying
\begin{align*}
    K \ge \frac{4\left(\phi(x_0) - \phi_{\inf} + \Gamma \sigma^2/c_b\right)}{\alpha c \epsilon^2}
\end{align*}
we have that
\eqref{eq.convergence measure} holds. Choosing the smallest positive integer satisfying this inequality gives the required upper bound on $K(\epsilon)$. Finally, using the iteration budget $K(\epsilon)$, we have
\begin{align*}
    S(\epsilon)
    &= K(\epsilon)\left\lceil c_b K(\epsilon)\right\rceil \le c_b K(\epsilon)^2+K(\epsilon).
\end{align*}
Substituting the bound on $K(\epsilon)$ yields the sample complexity.
\end{proof}

\section{Conclusion} \label{sec:4}

Our analysis extends the theoretical applicability of \textbf{SLAM}~\cite{Wang2025} by removing the assumptions of bounded true gradients and Lipschitz continuity of the nonsmooth term while preserving the original complexity orders under the remaining assumptions. \texttt{ChatGPT} assisted in developing the main idea of proof of Lemma~\ref{lem.gptproof}. This modified proof addresses the dependence of the accepted stepsize on the current sample batch by bounding sampling errors over a fixed finite set of possible backtracking stepsizes and comparing stochastic proximal updates with their {true counterparts that leverage true gradients}. With mini-batch sizes proportional to the iteration budget, the resulting iteration and sample complexities are $\mathcal{O}(\epsilon^{-2})$ and $\mathcal{O}(\epsilon^{-4})$, respectively, for obtaining an iterate whose expected proximal-gradient residual norm is at most $\epsilon$. These guarantees hold for any prescribed sequence of positive integer stepsize-resetting periods, allowing greater flexibility in period selection without modifying the algorithm.

\bibliographystyle{plainnat}
\bibliography{ref_nobgproof}
\end{document}